\documentclass[11pt]{article}
 \usepackage{graphicx}
 \usepackage{amssymb}
 \usepackage{epstopdf}
 \usepackage{amsmath}
 \usepackage{enumerate}
 \usepackage{amsthm}
 \usepackage{verbatim}
 \usepackage[all]{xypic}

 \usepackage{mystyle}

 \usepackage{ccfonts}

 \newtheoremstyle{mythmsty}{15pt}{5pt}{\it}{0pt}{\textbf}{}{0pt}{\textbf{(\thmnumber{#2})} \thmname{#1} \thmnote{ #3}
}

 \theoremstyle{mythmsty}

 \newtheorem{thm}[paragraph]{\bf Theorem}
 \newtheorem{cor}[paragraph]{\bf Corollary}
 
 \newtheorem{prop}[paragraph]{\bf Proposition}
 \newtheorem{lemma}[paragraph]{\bf Lemma}

\newtheorem{defn}[paragraph]{\bf Definition}
\newtheorem{theorem}[paragraph]{\bf Theorem}
\newtheorem{proposition}[paragraph]{\bf Proposition}

 \newenvironment{thm*}{\bf Theorem \it}{}
 \newenvironment{cor*}{\bf Corollary \it}{}
 \newenvironment{definition*}{\bf Definition \it}{}
 \newenvironment{prop*}{\bf Proposition \it}{}
 \newenvironment{lemma*}{\bf Lemma \it}{}
 \newenvironment{fact*}{\bf Fact \it}{}
 \newenvironment{rmk*}{\bf Remark \it}{}
 \newenvironment{exercise*}{\bf Exercise \it}{}
 \newenvironment{solution*}{\underline{Solution}\quad\rm}

  \numberwithin{equation}{subsubsection}

\newcommand{\mysec}[1]{\noindent \section{#1} \label{#1}}
\newcommand{\mysubsec}[1]{\noindent \subsection{#1} \label{#1}}
\newcommand{\p}[1]{\noindent\subsubsection{}\label{#1}\hskip -.1455in\textbf{)}}
\newcommand{\rf}[1]{(\ref{#1})}

\DeclareMathOperator{\Spec}{Spec}
\newcommand{\ang}[1]{\langle #1 \rangle}
\newcommand{\blank}{\mbox{$\underline{\makebox[10pt]{}}$}}
\DeclareMathOperator{\End}{End}
\DeclareMathOperator{\Hom}{Hom}
\DeclareMathOperator{\Id}{Id}
\DeclareMathOperator{\Ob}{Ob}
\DeclareMathOperator{\Spc}{Spc}

\newcommand{\mc}{\mathcal}
\newcommand{\sO}{\mc{O}}
\newcommand{\PP}{\mathbb P}

\title{\bf Prime spectra of derived quiver representations}

\author{{\small \textsc{ Yu-Han Liu and Susan J. Sierra}}}

\begin{document}
\date{\today}
\maketitle

\setcounter{tocdepth}{2}

\setcounter{section}{0}
\setcounter{subsection}{1}

\mysec{Introduction}

\mysubsec{Introduction}

\p{}  For every essentially small tensor triangulated category $T$, Balmer \cite{balmer_spectrumtt} defined functorially a ringed space $\mathrm{Spec}(T)$, called its prime spectrum.  It has been shown in some cases that the construction $T \mapsto \mathrm{Spec}(T)$ does not lose information in the sense that the category $T$ can be reconstructed from $\mathrm{Spec}(T)$.  For example, when $T_X=D(X)_\mathrm{parf}$ is the tensor triangulated category of perfect complexes on a scheme $X$, there is a natural morphism \[X\longrightarrow \mathrm{Spec}(T_X)\] of ringed spaces.  This comparison map is an isomorphism when $X$ is \emph{topologically noetherian}; see \cite[Theorem~6.3]{balmer_spectrumtt}.

\p{}  Not surprisingly,  Balmer's construction loses a great deal of information when applied to noncommutative rings, as we show in this note.  We give an alternate construction, showing how to construct an  algebra from a tensor triangulated category.  In some cases, our construction recovers much more information than Balmer's.

\p{}  Fix a field $k$.  Let $Q$ or $(Q,R)$ be a quiver \emph{with relations}; for legibility we will omit the relations $R$ from the notation $(Q,R)$ whenever convenient.  Let $Q\mathrm{-Rep}_k$ be its abelian category of finite dimensional representations over $k$.  Under suitable conditions on $R$ (see \rf{par:unital}), this category is equipped with a natural \emph{vertex-wise} tensor product which is an exact functor in each factor.  Its bounded derived category $D(Q):=D^b(Q\mathrm{-Rep})$ is then a tensor triangulated category when the relations satisfy  \rf{par:unital}.  We will describe $\mathrm{Spec}(D(Q))$ in the case when $Q$ is finite and ordered (that is, without non-trivial oriented cycles). 

It turns out that the ringed space $\mathrm{Spec}(D(Q))$ is not enough to recover $Q$, even in the case of quivers without relations; see \rf{par:examples}.  Moreover, $\mathrm{Spec}(D(Q))$ is not (yet) a quiver, hence we do not expect a comparison map \[Q\longrightarrow \mathrm{Spec}(D(Q)).\]

We show, however, that one can still recover $Q$ (or, more accurately, its path algebra) from the tensor triangulated category $D(Q)$, at least in the case that $Q$ is finite and ordered.  

\p{par:Pm}  These results were motivated in part by the fact that the derived categories of some projective varieties are equivalent to derived categories of quiver representations.   Consider the triangulated category $D(\mathbb P^m)=D^b(\mathrm{Coh}(\mathbb P^m))$.  It is known to be equivalent to the category $D(S_m)$, where $S_m$ is the quiver described in \cite[Example~5.3 and Example~6.4]{bondal_repassalgandcoh}: it has $m+1$ ordered vertices and $m+1$ arrows between consecutive vertices, with commutativity relations (see \rf{par:unital}).  Hence we have two different tensor products on the same triangulated category $T=D(\mathbb P^m)\cong D(S_m)$: one from the sheaf tensor product on $\mathbb P^m$ and the other the quiver tensor product.  We will see that the two tensor products give very different spectra.

{\bf Acknowledgements.}\;   The authors thank Ryan Kinser and Michael Wemyss for their helpful comments.  The second author was partially supported by an NSF Postdoctoral Research Fellowship, grant DMS-0802935.  

\mysubsec{Conventions and preliminary facts about quivers}

\p{} For a field $k$, $\mathbf{TT}_k$ denotes the category of (essentially small) $k$-linear tensor triangulated categories.  Morphisms in this category are exact functors preserving the tensor products and unit objects.

\p{}  If the base field $k$ is fixed,  we denote $\mathrm{Vect}_k$ simply by $\mathrm{Vect}$, and similarly $\mathbf{TT}_k$ by $\mathbf{TT}$ and $Q\mathrm{-Rep}_k$ by $Q\mathrm{-Rep}$.

\p{}   For any collection $M$ of objects in a tensor triangulated category $T$, $\langle M\rangle$ denotes the smallest \emph{tensor ideal} containing $S$.  (For us, a tensor ideal is a full, thick, triangulated subcategory $J$ of $T$ that satisfies $V \otimes W \in J$ for any $V \in J$ and $W \in T$.)  A tensor ideal $P$ is {\em prime} if $V \otimes W \in P$ implies that either $V \in P$ or $W \in P$.  

\p{par:quivers}  For definitions concerning quivers, see \cite{auslander_repartinalg}.  Let $v \in Q_0$. For any quiver representation $V$ we denote by $V_v$ the vector space associated to the vertex $v$. We take however the following convention that is different from \cite{auslander_repartinalg}:  We identify $Q\mathrm{-Rep}$ with the category of \emph{right} modules over the path algebra $kQ$, and take the corresponding convention on multiplication in $kQ$; for example, $e_m p = p$ for any path \emph{beginning} at the vertex $m$.  

 We consider mainly \emph{finite ordered quivers}, which means there are only finitely many vertices and arrows, and there are no oriented cycles.  In this case we denote the vertices by integers $n=1, 2, \ldots$, ordered in such a way that there are non-trivial paths from $n$ to $m$ only when $n<m$.
 
\p{par:unital}  We impose the condition on the relations $R$ of $(Q,R)$ that the the subcategory $(Q,R)\mathrm{-Rep}$ is a monoidal subcategory of $(Q,\emptyset)\mathrm{-Rep}$.  More explicitly, we require that the unit object $U$ in $(Q,\emptyset)\mathrm{-Rep}$ (over any field) lies in the subcategory $(Q,R)\mathrm{-Rep}$, and $(Q,R)\mathrm{-Rep}$ is closed under the quiver tensor product.  

We call relations involving only non-trivial paths satisfying this condition \emph{tensor relations}.

For example, relations generated by ``commutativity relations'' are tensor.  This means that there are relations $p_i-q_i\in R$, where $p_i$ and $q_i$ are non-trivial paths with the same source and target, such that every relation in $R$ is of the form $r=\sum \lambda_i(p_i-q_i)$ with $\lambda_i\in k$.  

\p{par:full}  For any quiver $Q$, a \emph{full subquiver} $Q'$ is a quiver with $Q'_0 \subseteq Q_0$ and $Q'_1 \subseteq Q_1$, that further inherits all arrows between vertices:  that is, for any arrow $a \in Q_1$ between vertices $n, m \in Q'_0$, we have $a \in Q'_1$.

\p{par:subq}  Let $Q$ be a quiver and $R$ a set of (not necessarily tensor) relations, which we assume to be an ideal in the path algebra $kQ$; let $Q'$ be a full subquiver of $Q$.  Denote by $f: Q'\rightarrow Q$ the inclusion map.  Then there is an exact \emph{restriction functor} \[f^*:Q\mbox{-Rep}\rightarrow Q'\mbox{-Rep}.\]

Denote by $R\cap Q'$ the subset of $R$ consisting of relations all of whose paths are in $Q'$; we view $R\cap Q'$ as a set of relations on $Q'$.  Then the image under $f^*$ of the subcategory $(Q,R)$-Rep is contained in $(Q',R\cap Q')$-Rep, and we have an induced exact functor \[f^*:(Q,R)\mbox{-Rep}\longrightarrow (Q',R\cap Q')\mbox{-Rep}.\]

We will need conditions under which its derived functor is full and essentially surjective.  To this end it is convenient to find a right-inverse.  Consider the exact \emph{extension by zero} functor \[f_*:Q'\mbox{-Rep}\rightarrow Q\mbox{-Rep}.\]  Observe that this functor respects the vertex-wise tensor product, but does not preserve the unit object in general.

Denote by $\bar R$ the set of relations on $Q'$ obtained by setting every arrow occurring in $R$ but not in $Q'$ to be zero.  Then the image under $f_*$ of $(Q',R\cap Q')$-Rep is contained in $(Q,R)$-Rep whenever $R\cap Q'=\bar R$.  (Notice that the containment $\subseteq$ always holds.)

We say that $R$ and $Q'$ are \emph{compatible} if the equality $R\cap Q'=\bar R$ holds (equivalently, $R \cap Q \supseteq \bar R$), in which case we have 

\centerline{\xymatrix{(Q,R)\mbox{-Rep}\ar@/^/[r]^-{f^*} & (Q',R')\mbox{-Rep},\ar@/^/[l]^-{f_*}  }}

where $R':=R\cap Q'=\bar R$.  These functors satisfy $f^*f_*=\mathrm{Id}$ on $(Q',R')$-Rep.

\p{par:qex}  For example if $Q$ is the quiver 
$
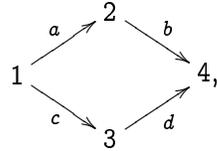
\xymatrix@R=10pt{& 2\ar[dr]^{b} & \\ 1\ar[ur]^{a}\ar[dr]_{c} && 4, \\ & 3\ar[ur]_{d}}$

$R=(ab-cd)$, and $Q'$ is the full subquiver $1\stackrel{a}{\longrightarrow} 2 \stackrel{b}{\longrightarrow} 4$, then $\bar R=\{ab\}$ and $R\cap Q'=\emptyset$.  Hence $R$ and $Q'$ in this example are not compatible.

On the other hand, notice that if $R$ has only non-trivial paths and $Q'$ consists of only one vertex then $R$ and $Q'$ are always compatible.

\p{par:tensorr}  Let $(Q,R)$ be a quiver with tensor relations, and $Q'$ a full subquiver compatible with $R$; let $R':=R\cap Q'=\bar R$.  We claim that $R'$  also consists of tensor relations.  First, $f^*$ sends the unit object $U$ in $Q$-Rep to the unit object $U'$ in $Q'$-Rep.  By assumption $U\in (Q,R)$-Rep, and so $U' = f^*(U)\in (Q',R')$-Rep.

Next, let $V',W'\in (Q',R')$-Rep; we need to show that $V'\otimes W'\in (Q',R')$-Rep too.  Since $f^*$ respects tensor products, we have \[V'\otimes W'=f^*f_*(V')\otimes f^*f_*(W')=f^*(f_*(V')\otimes f_*(W')).\]  But $f_*(V'), f_*(W')\in (Q,R)$-Rep, and so by assumption $f_*(V')\otimes f_*(W')\in (Q,R)$-Rep, whose image $V'\otimes W'$ under $f^*$ then lies in $(Q',R')$-Rep, as required.

\p{par:qwr}  Note that the full abelian subcategory $(Q,R)$-Rep of $Q$-Rep is in general not closed under extensions.  In particular there is an inclusion functor $D(Q,R)\rightarrow D(Q)$, but $D(Q,R)$ is not  a triangulated subcategory of $D(Q)$.  This inclusion  is an exact tensor functor (under the condition in \rf{par:unital}) that is injective on objects, but is not full in general:  For example, consider the quiver $Q$ with two vertices and two arrows $a$ and $b$ both going from one vertex to the other.  Let $R:=(a-b)$.  Then both simple objects in $Q$-Rep lie in $(Q,R)$-Rep, but most of their extensions do not lie in $(Q,R)$-Rep.

In general, all the simple objects of $Q$-Rep lie in $(Q,R)$-Rep, since all arrows act as the zero map in a simple representation.  

\p{par:filtration}  Let $(Q,R)$ be a finite ordered quiver with tensor relations (in particular, by the definition given in \rf{par:unital}, $R$ involves only non-trivial paths); let $q$ be the number of its vertices.  Denote its vertices by $1, 2, \ldots, q$ so that there exists a non-trivial path from $n$ to $m$ only when $n<m$.  (There is in general more than one such way to label the vertices.)
 
Let $Q_\ell$ be the full subquiver with vertices $\ell, \ell+1, \dots, q$, for every $\ell\leq q$.

\begin{lemma}\label{lem:comp}  $Q_\ell$ and $R$ are compatible for every $\ell$.  \end{lemma}

\begin{proof}  With $\ell$ fixed, let $\bar R$ be as defined in \rf{par:subq}, then we need to show $\bar R\subseteq R\cap Q_\ell$.  For any relation $r\in R$ denote by $\bar r\in \bar R$ the relation obtained by setting every arrow not in $Q_\ell$ to be zero.

The ideal of relations $R$ is generated by relations in $R$ of the form $r=\sum p_j$ with each $p_j$ a path between two fixed vertices, say from $n$ to $m$; in particular every arrow occurring in $r$ is between vertices $v,w$ with   $n \leq v < w \leq m$.

If $n \geq \ell$ then $r=\bar r$ is contained in $Q_\ell$, since in this case $Q_\ell$ contains every arrow occurring in $r$.

If $n<\ell$ then at least one arrow in each path occurring in $r$ does not lie in $Q_\ell$, namely the arrow with source $n$.  In this case $\bar r=0\in R\cap Q_\ell$ as well.  \end{proof}

With the ordering of the vertices fixed, we have a filtration $...\subset K_{\ell+1}\subset K_\ell \subset K_{\ell-1}\subset \ldots U$ of the unit object $U\in Q$-Rep, so that the quotient $K_\ell/K_{\ell+1}$ is the simple object $U(\ell)$ supported at the vertex $\ell$.  

\begin{cor}\label{cor:filtration}  If $(Q,R)$ is a finite ordered quiver with tensor relations, then every $K_\ell$, $\ell=1, 2, \ldots, q$, lies in the subcategory $(Q,R)$-Rep.\end{cor}

\begin{proof}  Notice that if $f$ denotes the inclusion of $Q_\ell$ in $Q$, then $K_\ell$ is the image under $f_{*}$ of the unit object $U_\ell$ in $Q_\ell$-Rep.  Let $R_\ell:=R\cap Q_\ell$.  Then by \rf{lem:comp} and \rf{par:tensorr} we know that $(Q_\ell,R_\ell)$ is a finite ordered quiver with tensor relations; in particular $U_\ell\in (Q_\ell,R_\ell)$-Rep.  By \rf{lem:comp} and \rf{par:subq} we have $K_\ell=f_*U_\ell\in (Q,R)$-Rep.  \end{proof}

\mysec{Balmer's constructions and their applications to quivers}

In this section, we recall briefly the constructions in \cite[Definition~2.1 and Definition~6.1]{balmer_spectrumtt}, and apply them to quiver representations.  We see that a great deal of information is lost.

\mysubsec{Spectrum of a tensor triangulated category}

\p{} Let $T$ be a tensor triangulated category.  Here we recall Balmer's definition of the topological space $\Spc T$.  

\begin{defn}\label{def:spec}
 The space $\Spc T$ is the set of prime $\otimes$-ideals of $T$.  It is given the {\em Zariski topology}:  a closed set is of the form
\[ \mathbf{Z}(S) := \{ P \in \Spc T \vert S \cap P = \emptyset \},\]
where $S$ is a set of objects in $T$.
\end{defn}

\p{} If $X$ is a (topologically) noetherian scheme, by \cite[Corollary~5.6]{balmer_spectrumtt}, the topological space $\Spc(D(X)_\mathrm{parf})$ is homeomorphic to $X$.

 \p{} For any quiver $(Q,R)$ with tensor relations (see \rf{par:unital}) let $D(Q,R)$, and often just $D(Q)$, be the bounded derived category of $(Q,R)$-Rep, equipped with the vertex-wise tensor product.  
 
 We now calculate $\Spc D(Q)$ for a finite ordered quiver $Q$ with tensor relations, and show that the functor $Q \mapsto \Spc D(Q)$ loses a great deal of information; in fact, it retains merely the number of vertices of $Q$!

\p{}  Recall that every object $V$ in $D(Q)$ can be represented by a bounded complex of objects in $Q\mathrm{-Rep}$.  In particular its cohomology $H(V)=\bigoplus H^i(V)$ is again an object in $Q\mathrm{-Rep}$.  We call the graded vector space $H(V)_n$ \emph{the cohomology of $V$ at the vertex $n$}.

\begin{lemma}\label{lem:main}  Let $Q$ be a finite ordered quiver with tensor relations.  For any object $V$ in $D(Q)$, we have $\langle V\rangle=\langle U(n)\,|\, H(V)_n\neq 0\rangle$, where $U(n)$ is the simple representation corresponding to the vertex $n$.\end{lemma}

\begin{proof}  The unit object is given by the representation $U$ with $U_n\cong k$ at each vertex $n$ and identity maps between them for all arrows.  Since $Q$ is finite and ordered, $U$ admits a filtration whose successive quotients are the simple representations $U(n)$; see \rf{par:filtration}. 

Tensoring this filtration with $V$ we see that $\langle V\rangle$ contains the object $V(n)=V\otimes U(n)$ with $V(n)_n=V_n$ and $V(n)_m=0$ for $m\neq n$; all the arrows in the representation $V(n)$ are the zero map.  On the other hand, $V$ is an extension of the $V(n)$'s; thus $V \in \ang{V(n)}$.  We have obtained that $\langle V\rangle=\langle V(n)\,|\, n=1, 2,\ldots\rangle$.

The object $V(n)\in \langle V\rangle$ is a complex of vector spaces, and the construction in \cite[Chapter~III, \S~1.4 Proposition]{gelfand_manin_homologicalalgebra} shows that it is isomorphic in $D(Q)$ to the complex $H(V)_n$ of vector spaces with zero differentials.  Since $\langle V\rangle$ is a thick triangulated subcategory,  it contains the simple representation $U(n)$ whenever $H(V)_n\neq 0$. 
  \end{proof}

\begin{cor}\label{cor:1}  If $V$ can be represented by a complex with non-zero cohomology at every vertex, then $\langle V\rangle=D(Q)$ is the unit ideal.\qed \end{cor}

\p{H_n}   Consider the evaluation functor $V\mapsto V(n)=V\otimes U(n)$ from $Q\mathrm{-Rep}$ to the category $\mathrm{Vect}$.  This functor is exact and preserves the tensor products as well as the unit object.  Its derived functor from $D(Q)$ to $D^b(\mathrm{Vect})\cong \oplus \mathrm{Vect}[j]$ then sends $V$ to its cohomology $H(V)_n$ at $n$.  The kernel of this derived functor is a tensor ideal, which we denote by $P_n$.  It is the full subcategory of $D(Q)$ consisting of objects $V$ with $H(V)_n=0$.

\begin{theorem}\label{thm:main}  Let $Q$ be a finite ordered quiver with tensor relations.  Then $\Spc D(Q)$ is the discrete space $\{ P_n \vert \, n \in Q_0\}$.  \end{theorem}

\begin{proof}  Clearly $P_n \supseteq \langle U(m) \vert m \neq n \rangle$.  Suppose that $V \not \in P_n$.  
   Then $H(V)_n\neq 0$, hence $\displaystyle V\oplus\bigoplus_{m\neq n}U(m)$ is an object in $\langle P_n, V\rangle$ which has non-zero cohomology at every vertex.  By \rf{cor:1} we have $\langle P_n, V\rangle=D(Q)$.  Thus $P_n$ is maximal, and by  \cite[Proposition~2.3(c)]{balmer_spectrumtt} $P_n$ is prime.  

Let $I$ be an ideal that is not contained in any $P_n$.  Then for every $n$ we can find $V^n\in I$ with $H(V^n)_n\neq 0$.  This implies that $\oplus V^n\in I$ has non-zero cohomology at every vertex, and $I$ must then be the unit ideal.  Thus the $P_n$ are precisely the maximal ideals of $D(Q)$.

Let $P$ be a prime ideal in $D(Q)$.  By the previous paragraph it is contained in $P_n$ for some $n$, in particular $U(n)\notin P$.  But $U(m)\otimes U(n)=0\in P$ whenever $m\neq n$, hence $U(m)\in P$ for every $m\neq n$ since $P$ is prime.  That is, $P=P_n$.\end{proof}

\p{} Let $T$ be the triangulated category $D^b(\mathbb P^m)$ which by \rf{par:Pm} is equivalent to $D^b(S_m)$.  By \cite[Corollary~5.6]{balmer_spectrumtt} the spectrum $\mathrm{Spc}(T,\otimes_{\mathbb P^m})$ under the sheaf tensor product is homeomorphic to $\mathbb P^m$, while by \rf{thm:main} above $\mathrm{Spc}(T,\otimes_{S_m})$ is $m+1$ discrete points.  The same phenomenon happens for example for Grassmannians by \cite{kapranov_dercatcoh}, and more generally for varieties admitting a full exceptional set of objects. 

\mysubsec{The structure sheaf}

\p{support}  Let $T$ be a triangulated tensor category.  According to \cite[Definition~2.1]{balmer_spectrumtt} the space $\mathrm{Spc}(T)$ comes with a map $\mathrm{supp}(-)$ from the objects of $T$ to closed subsets of $\mathrm{Spc}(T)$ associating to every object $V$ in $T$ the set of prime ideals {\em not} containing $V$.   

By comparing this with \rf{thm:main} we see that for any object $V$ of $D(Q)$ we have \[\mathrm{supp}(V)=\{P_n\,|\,H(V)_n\neq 0\}.\]

\p{} Recall from \cite[Definition~6.1]{balmer_spectrumtt} the definition of the structure sheaf on $\mathrm{Spc}(T)$:  For any open subset $W\subseteq \mathrm{Spc}(T)$, denote by $Z$ its complement, and $T_Z$ the full triangulated subcategory of $T$ consisting of objects $a$ with $\mathrm{supp}(a)$ contained in $Z$.  Then the localization functor \cite{verdier_derivedcategories} $T\rightarrow T/T_Z$ is a tensor functor between tensor triangulated categories.  Denote still by $U$ the image of the unit object $U\in T$ in $T/T_Z$, then $\mathcal O_{\mathrm{Spec}(T)}$ is by definition the sheafification of the presheaf of rings
\begin{equation}\label{presheaf}
W\mapsto \mathrm{End}_{T/T_Z}(U).
\end{equation}

The {\em spectrum} of a tensor triangulated category $T$ is the pair $\Spec T := (\Spc T, \sO_{\Spec(T)})$.  

\p{} If $X$ is a topologically noetherian scheme, then $\Spec(D(X)_\mathrm{parf}) \cong X$ as locally ringed spaces, by \cite[Theorem~6.3]{balmer_spectrumtt}.   

\p{} We now consider the same construction for $T = D(Q)$.  

\begin{theorem}\label{thm:structure sheaf quiver}
 Let $Q$ be a finite ordered quiver with tensor relations.  Then $\sO_Q := \sO_{\Spec(D(Q))}$ is the constant sheaf of algebras $k$. That is, for any $W \subseteq \Spc(D(Q))= Q_0$, we have $\sO_{Q}(W) \cong k^{\oplus W}$.  
\end{theorem}

\begin{proof}  Since $\mathrm{Spc}(D(Q))$ is a discrete topological space, it suffices to show that $\mathcal O_Q(\{v\})\cong k$ on the open set $\{v\}$ consisting of one point.  Let $Z:=\mathrm{Spc}(D(Q))-\{v\}$ be the complement of $\{v\}$ and $Q'$ the full subquiver with only one vertex $v$.  Denote by $f: Q'\rightarrow Q$ the inclusion.  By identifying $D(Q')$ with $D^b(\mathrm{Vect})$ we see that $f^*$ is exactly the functor $V\mapsto H(V)_v$ on $D(Q)$.  In particular we have $T_Z=\ker(f^*)$ and an induced functor \[\bar f^*: D(Q)/T_Z= D(Q)/\ker(f^*)\longrightarrow D(Q').\]

Since $Q'$ is compatible with tensor relations (see \rf{par:qex}), we may apply the next proposition to conclude the proof.  \end{proof}

\begin{prop}  Let $(Q,R)$ be a finite ordered quiver with relations.  Let $f:Q'\rightarrow Q$ be the inclusion of a full subquiver compatible with the relations; let $R':=R\cap Q'$ as in \rf{par:subq}.  Then the derived restriction functor $f^*:D(Q)\rightarrow D(Q')$ induces an equivalence \[\bar f^*: \overline T:=D(Q)/\ker(f^*)\longrightarrow D(Q').\]

\end{prop}

\begin{proof}  We claim that $f^*$ is essentially surjective and full:  Indeed, the derived functor of the extension by zero exact functor \[f_*:Q'\mathrm{-Rep}\rightarrow Q\mathrm{-Rep}\] is a right inverse of $f^*$ on both objects and morphisms (see \rf{par:subq}).  It follows that $\bar f^*$ is also essentially surjective and full, and so it remains to show that it is faithful.  

Let $X, Y \in \Ob(D(Q)) = \Ob(\overline T)$.  An element of $\Hom_{\overline T}(X,Y)$ is represented by a diagram  $X\stackrel{s}{\leftarrow}V\stackrel{g}{\rightarrow}Y$ in $T$,
where $s$ is such that $f^*s$ is an isomorphism.  This is thought of as a ``morphism'' $g s^{-1}:  X \to Y$. 

Now let $X\stackrel{s}{\leftarrow}V\stackrel{g}{\rightarrow}Y$ represent a morphism in $\overline T$ that maps to zero under $\bar f^*$. 
In particular $g:V\rightarrow Y$ is a morphism in $D(Q)$ such that $f^*g=0$.  By applying $f^*$ to the distinguished triangle \[V\stackrel{g}{\longrightarrow} Y\stackrel{h}{\longrightarrow} \mathrm{cone}(g)\] we see that this implies $f^*h$ has a left inverse $m:f^*\mathrm{cone}(g)\rightarrow f^*Y$ in $D(Q')$.  

Let $\tilde m: \mathrm{cone}(g)\rightarrow Y$ be in $D(Q)$ such that $f^*\tilde m$ is equal to $m$; such an $\tilde m$ exists since $f^*$ is full, as observed above.  

Now $f^*(\tilde m\circ h)=m\circ f^*h$ is an isomorphism in $D(Q')$.  Therefore $\mathrm{cone}(\tilde m\circ h)$ lies in $\ker(f^*)$ and this means that $\tilde m \circ h$ also becomes an isomorphism in $\overline T$.  Hence $g$ maps to zero in $\overline T$ since $(\tilde m\circ h)\circ g$ maps to zero in $\overline T$.  So $X\stackrel{s}{\leftarrow}V\stackrel{g}{\rightarrow}Y$ represents the zero morphism in $\overline T$. \end{proof}

\p{} The presheaf \rf{presheaf} formally contains  more information than its sheafification $\sO_{\Spec(T)}$. In the case of a quiver, this difference is significant. Let $v, w$ be vertices in a quiver $Q$ without relations, and let $W := \{ v, w\}$ considered as an open subset of $\mathrm{Spec}(D(Q))$; denote by $Q_W$ the full subquiver of $Q$ consisting of vertices $v$ and $w$.  

Then $\End_{D(Q_W)}(U)$ is either $k$ or $k \oplus k$, depending on whether there are arrows between $v$ and $w$.  Thus the presheaf \rf{presheaf} recovers the underlying graph of the quiver, while the structure sheaf recovers only the number of vertices.  

\p{prime ideals}  \emph{Prime ideals in the path algebra:}  By \cite[page~53]{auslander_repartinalg} we know that the path algebra $kQ$ modulo its radical (which is the ideal generated by all the non-trivial paths) is isomorphic to the product of $\#Q_0$ copies of $k$.  In particular we see that the two-sided prime ideals in $kQ$ are naturally in bijection with prime ideals in $D(Q)$.  That is, the global sections of $\sO_{\mathrm{Spec}(D(Q))}$ are naturally isomorphic to $k Q/J(kQ)$, and, since the map $kQ \to kQ/J(kQ)$ has a right inverse, to a subalgebra of $kQ$.  

\p{par:examples}   Let $Q_i$, $i\geq 1$, be the $i$-Kronecker quiver:  the quiver (without relations) with two vertices and $i$ arrows all going from one vertex to the other.  There are obvious morphisms $Q_i\rightarrow Q_j$ with $i<j$, inducing functors $D(Q_j)\rightarrow D(Q_i)$ of tensor triangulated categories.  These in turn induce morphisms \[\mathrm{Spec}(D(Q_i))\longrightarrow \mathrm{Spec}(D(Q_j))\] of ringed spaces.  From \rf{thm:main} we see easily that these are isomorphisms, in particular we cannot recover quivers from their prime spectra.  Note that the presheaves \rf{presheaf} are isomorphic for all $Q_i$ as well.

\mysec{Functors of points}

\setcounter{subsection}{1}
\setcounter{subsubsection}{0}

%no subsection

\p{}  Let $T, S \in \mathbf{TT}$.    Define \[T(S):=(\mathrm{Hom}_{\mathbf {TT}}(T,S)/\cong\,)=(\mathrm{Fun}^{\otimes,1}(T,S)/\cong).\]
Elements in $T(S)$ will be called \emph{$S$-valued points in $T$}.  We then have a set-valued covariant functor $T(-)$ on $\mathbf{TT}$ for every $T\in\mathbf{TT}$.

\p{}  If $R$ is a commutative $k$-algebra, we define $T(R):= T(D^b(R-{\rm Mod}))$. Elements in $T(R)$ will be called \emph{$R$-rational points in $T$}. 

\p{} Let $Q$ be a finite ordered quiver with tensor relations \rf{par:unital}; we will compute $D(Q)(k)$.  We identify the tensor triangulated category $D^b(k) = D^b(\mathrm{Vect})$ with $\oplus \mathrm{Vect}[j]$ by taking cohomology of complexes of vector spaces.

 The functors $V\mapsto H(V)_n$ in \rf{H_n} are in $D(Q)(k)$, and conversely:

\begin{prop}\label{prop: points of quiver}  $D(Q)(k)=\{(V\mapsto H(V)_n)\,|\, n\in Q_0\}$.
\end{prop}

\begin{proof}  Let $F\in D(Q)(k)$.  Consider the filtration $\ldots \subset K_{n+1}\subset K_{n}\subset \ldots\subset U$ of the unit object in $D(Q)$ such that the quotient $K_n/K_{n+1}$ is the simple object $U(n)$ supported at the vertex $n$.  By  \rf{cor:filtration}, each $K_n$  lies in $D(Q)=D(Q,R)$. 

Since $F$ preserves the unit objects, $F(U)$ is isomorphic to $k$ viewed as a graded vector space placed at degree $0$.  Hence on applying $F$ to the distinguished triangle \[K_{n+1}\stackrel{f_{n+1}}{\longrightarrow} U\stackrel{g_{n+1}}{\longrightarrow} U/K_{n+1},\] we see that for any fixed vertex $n$ exactly one of $F(f_{n+1})$ and $F(g_{n+1})$ is equal to $0$, while the other one admits a one-sided inverse. 

 Note that for $n$ sufficiently large we have $F(f_{n+1})=0$.  
Moreover, if $F(f_{n+1})=0$ then $F(f_{n+2})=0$:  Indeed, suppose on the contrary that $F(f_{n+1})=0$ but $F(f_{n+2})\neq 0$, then we get a contradiction by applying $F$ to the following commutative diagram:

\centerline{\xymatrix{K_{n+2} \ar[r]^-{f_{n+2}} \ar[d] & U \ar@{=}[d]\\ K_{n+1}\ar[r]^-{f_{n+1}}& U.}}

So from now on let $n$ be minimal so that $F(f_{n+1})=0$.  Consider the following diagram of distinguished triangles, where by assumption $F(f_n)$ is non-zero and admits a right inverse $m$:

\[
\xymatrix{K_{n+1} \ar[r]^-{f_{n+1}}\ar[d] & U \ar[r]\ar@{=}[d] & U/K_{n+1}\ar[d] &&& F(K_{n+1}) \ar[r]^-{0 }\ar[d] & F(U) \ar[r]\ar@{=}[d] & F( U/K_{n+1})\ar[d]
 \\ K_n \ar[r]^-{f_n}\ar[d]_-q & U \ar[r]\ar[d] & U/K_n \ar[d] & \ar @/^/ @{..>} [r]^{F} && F(K_n) \ar[r]^{F(f_n)}\ar[d]_{F(q)} & F(U) \ar@/^1pc/[l]^{m} \ar[r]\ar[d] & F(U/K_n) \ar[d] \\ 
U(n) \ar[r] & 0 \ar[r] & U(n)[1] &&& F(U(n)) \ar[r] & 0 \ar[r] & F(U(n))[1].}\]

  The map $F(q) \circ m: F(U)\rightarrow F(U(n))$ must  be non-zero, since otherwise $m$ lifts to a non-zero map from $F(U)$ to $F(K_{n+1})$, giving a right inverse to $F(f_{n+1})$, which is zero by assumption.

This says that $F(U(n))$ contains $F(U)=k$ as a direct summand, and when viewed as an object in $\oplus \mathrm{Vect}[j]$, $F(U(n))$ must be isomorphic to $k$ since $U(n)\otimes U(n)=U(n)$.

The two functors $V\mapsto F(V)=F(U\otimes V)$ and $V\mapsto F(U(n)\otimes V)=F(V_n)$ are then isomorphic:   
\[ F(V) = F(U \otimes V) = k \otimes F(V) = F(U(n))\otimes F(V) = F(U(n)\otimes V) = F(V_n).\]

Now the full subcategory of $D(Q)$ consisting of objects of the form $U(n)\otimes V$ is isomorphic to $D^b(k)$, and any exact functor preserving tensor products and the unit objects from $D^b(k)$ to itself is the identity.  Hence on identifying $D^b(k)$ with $\oplus \mathrm{Vect}[j]$ we see that the functor $F$ is isomorphic to $V\mapsto H(V)_n$.  \end{proof}

\p{par:sum}  To summarize, we have established bijections for any finite ordered quiver $Q$ with tensor relations:

\hskip 1.5in\xymatrix{Q_0\ar[r] & D(Q)(k)\ar[r] & \mathrm{Spc}(D(Q)) \\ n \ar@{|->}[r] & (V\mapsto H(V)_n) \\ & F \ar@{|->}[r] & \ker F.}

\p{par:F_n} Let $F_n \in D(Q)(k)$ be the functor $V\mapsto H(V)_n$ corresponding to the vertex $n$.   We now give a useful formula for the functors $F_n$.  Denote by $M_n$ the indecomposable projective module of the path algebra $kQ$ of $Q$ with simple quotient $U(n)$, namely, it is the \emph{right} submodule of $kQ$ generated by the trivial path $e_n$ at the vertex $n$:  $M_n=e_n(kQ)$ is the $k$-vector space spanned by all paths starting at the vertex $n$; see \cite[Page~59]{auslander_repartinalg}.  Then \[kQ\cong \bigoplus_n M_n.\]  Under the usual identification of $Q\mathrm{-Rep}$ with $\mathrm{mod-}kQ$, we may view $M_n$ as an object in the former abelian category.

\begin{lemma}\label{lem:F_n}  We have $F_n(-)=R\mathrm{Hom}(M_n,-)$ on $D(Q)$.  

\end{lemma}
\begin{proof}  \begin{comment}Recall that $F_n$ may be given by the derived functor on $D(Q)$ of the \emph{exact} functor $-\otimes U(n)$ on $Q\mathrm{-Rep}$.  Since $\mathrm{Hom}(M_n,-)$ is also exact, it suffices to show that $\Hom(M_n, \blank) = F_n(\blank)$ on $Q\mathrm{-Rep}$.

Let $V$ be an object in $Q\mathrm{-Rep}$, and $\tilde V$ the $kQ$-module corresponding to it.  Then $V\otimes U(n)=V_n=\tilde Ve_n$.  On the other hand, the evaluation map from $\mathrm{Hom}(M_n,\tilde V)$ to $\tilde V$, sending $f\mapsto f(e_n)$, is clearly an injection since $M_n=e_n(kQ)$ and $f$ is a right $kQ$-module homomorphism.  

The image of this evaluation map is exactly $\tilde Ve_n$, since $f(e_n)=f(e_n^2)=f(e_n)e_n$, and for any $xe_n\in \tilde Ve_n$, the assignment $f:e_n\mapsto xe_n$ extends to an element in $\mathrm{Hom}(M_n,\tilde V)$. \end{comment}
This is well-known, and follows for example from \cite[Exercise~III.10]{auslander_repartinalg}. \end{proof}

\mysec{Algebras associated to a tensor triangulated category}

\mysubsec{Defining the algebra}

\p{}  Let $S, T\in\mathbf{TT}$. For any $F,G\in T(S)$ denote by $\mathrm{Hom}(F,G)$ the set of natural transformations.  Since $T$ is a $k$-linear category, this is a $k$-vector space.
 Let $U$ be the unit object of $T$.   

On the $k$-vector space \[A(T,S):=\prod_{F,G\in T(S)}\mathrm{Hom}(F,G)\] we then have a partially defined product by composition, which gives an associative, in general non-commutative algebra by defining the product between elements which are not composable to be zero.  We are most interested in the special case:
\[ A(T) := A(T, D^b(k)) := \prod_{F,G\in T(k)}\mathrm{Hom}(F,G).\]

\p{} The functor  $A(\blank, \blank)$ is contravariant in its first argument and covariant in its second.

\mysubsec{Recovering finite ordered quivers}

\p{}  We now show that a finite ordered quiver with tensor relations \rf{par:unital} can be essentially recovered from its tensor triangulated category of representations.  Recall that for a quiver $Q$ we denote by $D(Q)$ the bounded derived tensor category of finite dimensional representations of $Q$ over the fixed field $k$.

\begin{thm}\label{thm:recovery}  Let $(Q, R)$ be a finite ordered quiver with tensor relations.  The  algebra $k Q/(R)$ of $(Q, R)$ over $k$ is naturally isomorphic to $A(D(Q))$.
\end{thm}

\begin{proof}  Recall that we have a bijection (see \rf{par:sum}) from the set of vertices $Q_0$ to $T(k)$ given by $n\mapsto (F_n:V\mapsto H(V)_n)$.  

Let $\Lambda:= kQ/(R)$.  
The set of paths from vertex $n$ to $m$ is $e_n\Lambda e_m$ (where $e_n$ is the trivial path at $n$), and we first define a map \[\phi:e_n\Lambda e_m \longrightarrow \mathrm{Hom}(F_n,F_m)\] for every pair of vertices $n,m$.  Let $p\in e_n\Lambda e_m$ be a path from $n$ to $m$, and $V\in D(Q)$, then we define \[\phi(p)_V: F_n(V)\rightarrow F_m(V)\] to be the morphism $H(V)(p): H(V)_n\rightarrow H(V)_m$ (where the cohomology $H(V)$ is a complex of $Q$-representations).  This is well-defined since $H(V)(r)=0$ for every relation $r$.  Note that this is defined without assumptions on  $Q$ or $R$.

Now by the Yoneda lemma on the category $D(Q)$, remembering \rf{lem:F_n} that $F_n$ is represented by $M_n=e_n\Lambda $, we have \[\mathrm{Hom}(F_n,F_m)\cong F_m(M_n)\cong \mathrm{Hom}_{\Lambda }(M_m, M_n),\]sending $\rho: F_n\rightarrow F_m$ to $\rho_{M_n}(e_n)$, where $\rho_{M_n}:(F_n(M_n)=ke_n)\rightarrow F_m(M_n)$.

On the other hand we have a natural isomorphism \[\psi:\mathrm{Hom}_\Lambda (M_m,M_n)\longrightarrow e_n\Lambda e_m,\] sending $f: M_m\rightarrow M_n$ to $f(e_m)\in M_n$; see  \rf{lem:F_n}.  It is now easy to show that $\psi$ is the inverse of $\phi$: 

\centerline{\xymatrix{e_n\Lambda e_m \ar[r]^-\phi & \mathrm{Hom}(F_n, F_m) \ar[r] & F_m(M_n) \ar[r] & \mathrm{Hom}(M_m, M_n) \ar[r]^-{\psi} & e_n\Lambda e_m   \\ p \ar@{|->}[r] & (V\mapsto \phi(p)_V) \ar@{|->}[r]& \phi(p)_{M_n}(\mathrm{Id}_{M_n} ) \ar@{|->}[r] & (e_mq\mapsto pe_mq) \ar@{|->}[r] & pe_m=p.}     }

The naturality in the statement means the following.  Let $(Q, R)$ and $(Q', R')$ be two quivers with tensor relations; let $\Lambda:= kQ/(R)$ and let $\Lambda':= kQ'/(R')$.  Suppose there is a $k$-algebra homomorphism $\Lambda \to \Lambda'$.  We obtain an induced restriction functor $D(Q', R') \to D(Q, R)$ and a homomorphism $A(D(Q, R)) \to A(D(Q', R'))$.  

Then these homomorphisms form a commutative square with the isomorphisms $\Lambda \cong A(D(Q))$ and $\Lambda'\cong A(D(Q'))$ above.  We leave the details to the reader. \end{proof} 

Thus a finite, ordered quiver $Q$ with tensor relations (or at least its path algebra) can be recovered from its tensor triangulated category $D(Q)$ of representations.  By \cite[Chapter~III, Theorem~1.9(c)(d)]{auslander_repartinalg}, the quiver $Q$ can be recovered from its path algebra if the ideal generated by its relations $R$ lies between $J^2$ and $J^t$ for some integer $t$, where $J$ is the radical of $kQ$.

\mysubsec{Comparing the two constructions}

\p{}  For  a tensor triangulated category $T$, the structure sheaf $\sO_{\Spec T}$ and the algebra $A(T)$ are naturally related.  This follows from:
\begin{proposition}\label{prop:algebra}
 Let $T, S$ be  tensor triangulated categories, and let $U$ be the  unit object of $T$.  Then $A(T, S)$ is naturally an $\End_T(U)$-algebra. 
\end{proposition}
\begin{proof}
 Let $\phi \in \End_T(U)$.  Then $\phi$ induces a natural transformation $\Phi:  \Id_T \to \Id_T$, where $\Phi_M:  M \to M$ is given by
\[ \xymatrix{ M \ar[r]^{\cong} & M \otimes U \ar[r]^{1 \otimes \phi} & M \otimes U \ar[r]^{\cong}&  M.}\]

For every $F \in T(k)$, this then induces a natural transformation $F(\Phi):  F\to F$, defined by
\[ F(\Phi)_M = F(\Phi_M):  FM \to FM.\]
The map
\begin{align*}
 z:  \End_T(U) & \to A(T, S) \\
 \phi & \mapsto (F(\Phi))_F
\end{align*}
 is easily seen to be a ring homomorphism.

Let $F, G \in T(S)$ and let $\beta \in \Hom(F, G)$.  Let $M \in T$.  Then by naturality of $\beta$, the diagram
\[ \xymatrix{
 FM \ar[r]^{F(\Phi_M)} \ar[d]_{\beta_M} & FM \ar[d]^{\beta_M} \\
GM \ar[r]_{G(\Phi_M)} & GM }
\]
commutes.  That is, $G(\Phi) \circ \beta = \beta \circ F(\Phi)$ in $A(T, S)$, and $z(\phi) \in Z(A(T))$.
\end{proof}

\p{} From the previous result, we see that for any tensor triangulated category $T$, there are ring homomorphisms
\[ \xymatrix{
 \End_T(U) \ar[r]^{\alpha} \ar[rd]_{z} & \Gamma(\sO_{\Spec T}) \\
& A(T), }
\]
where $\alpha$ is induced from the canonical map from a presheaf to the associated sheaf.  
In general, there does not seem to be any reason why there should be a vertical map (in either direction) completing the triangle.  However, we have:

\begin{proposition}\label{prop:quiver maps}  Let  $Q$ be a finite ordered quiver with tensor relations and $T = D(Q)$.  Then there are vertical maps so that the diagram
\[ \xymatrix{
 \End_T(U) \ar[r]^{\alpha} \ar[rd]_{z} & \Gamma(\sO_{\Spec T}) \ar@<1ex>[d] \\
& A(T) \ar@<1ex>[u],}
\]
commutes; further $\End_T(U) \cong Z(A(T))$.
\end{proposition}
\begin{proof}
 It is easy to see that $\End_T(U) \cong k^{\pi_0(Q)}$ is the center of the path algebra $kQ$; here $\pi_0$ denotes the set of equivalence classes of vertices in $Q$ generated by the existence of arrows between them.  The vertical maps come from \rf{prime ideals}.
\end{proof}

\p{} We note that for tensor triangulated categories $T = D^b(X)$ induced from a scheme $X$, the algebra $A(T)$ can be quite unpleasant, and does not necessarily recover $X$.  For example, let $k$ be an algebraically closed field, and let $T := D^b(\PP^1_k)$.  Then  $T(k) = \PP^1(k)$; that is, the only tensor functors from $T$ to $D^b(k)$ are given by restriction to a $k$-point.  The algebra $A(T)$ is then easily seen to be
\[ A(T) = \prod_{p \in \PP^1(k)} k_p,\]
the direct product of the skyscraper sheaves.  Indeed, it suffices to show that if $x,y$ are in $\mathbb P^1(k)$, then there are no non-zero natural transformations between the corresponding functors $x^*$ and $y^*$.  

Suppose on the contrary that $\phi$ is a non-zero natural transformation from $x^*$ to $y^*$.  Then there is a coherent sheaf $\mathcal F$ such that both fibres $x^*\mathcal F$ and $y^*\mathcal F$ are non-zero, and $\phi_\mathcal F$ is a non-zero map between them.  In fact since $D^b(\mathbb P^1)$ is generated by $\mathcal O$ and $\mathcal O(1)$, we may assume that $\mathcal F$ is a line bundle.  In this case the non-zero map $\phi_\mathcal F$ must be an isomorphism between the fibres $x^*\mathcal F$ and $y^*\mathcal F$.

Now let $\mathcal F':=y_*y^*\mathcal F$.  Then we have $x^*\mathcal F'=0$ but $y^*\mathcal F\rightarrow y^*\mathcal F'= y^*y_*y^*\mathcal F$ is a non-zero map.  Hence we obtain a commutative diagram, by the naturality of $\phi$:

\hskip 2.4in\xymatrix{x^*\mathcal F \ar[r]\ar[d]_-{\phi_\mathcal F}^-\cong & x^*\mathcal F'=0 \ar[d]_-{\phi_{\mathcal F'}}\\ y^*\mathcal F\ar[r]^-{\neq 0} & y^*\mathcal F'.}

But this implies that $\phi_\mathcal F=0$, a contradiction.

Thus for triangulated tensor categories coming from algebraic geometry, Balmer's construction is much better than ours. 
It would be interesting to find a functorial construction that combines the good features of both and to prove a reconstruction theorem that generalizes simultaneously \rf{thm:recovery} and \cite{balmer_spectrumtt}.

\bibliographystyle{plain}
\bibliography{mybibli}

\end{document}